\documentclass[12pt,a4paper,reqno]{amsart}
\usepackage{amsfonts}
\usepackage{amsthm}
\usepackage{amsmath}
\usepackage{amscd}
\usepackage[latin1]{inputenc}
\usepackage{t1enc}
\usepackage[mathscr]{eucal}
\usepackage{indentfirst}
\usepackage{graphicx}
\usepackage{graphics}
\usepackage{subfigure}
\usepackage{pict2e}
\usepackage{epic}
\numberwithin{equation}{section}
\usepackage[margin=2.9cm]{geometry}
\usepackage{epstopdf} 
\usepackage[colorlinks=true,linkcolor=blue,urlcolor=black]{hyperref}

\theoremstyle{plain}
\newtheorem{teo}{Theorem}[section]
\newtheorem{lemma}[teo]{Lemma}
\newtheorem{cor}[teo]{Corollary}
\newtheorem{prop}[teo]{Proposition}
\newtheorem{obs}[teo]{Observation}
\newtheorem{rem}{Remark}

 \theoremstyle{definition}
\newtheorem{defi}[teo]{Definition}

\newtheorem{?}[teo]{Problem}

\newcommand{\R}{\mathbb{R}}
\newcommand{\Q}{\mathbb{Q}}
\newcommand{\F}{\mathcal{F}}
\newcommand{\PP}{\mathcal{P}}
\newcommand{\RR}{\mathcal{R}}
\newcommand{\T}{\mathbb{T}} 
\newcommand{\N}{\mathbb{N}}
\newcommand{\Su}{\mathbb{S}}
\newcommand{\ds} {\displaystyle}
\newcommand{\supp} {{\rm{supp}}}

\begin{document}

\title{Generic H\"{o}lder foliations with smooth leaves}

\author[E. Fuentes]{Enzo Fuentes}

\address{Instituto de Matem\'aticas \\ Facultad de Ciencias \\
Pontificia Universidad Cat\'olica de Valpara\'iso} 

\email{enzo.fuentes.m@gmail.com}



\begin{abstract} In this work, we consider a specific space of foliations with $C^1$ leaves and H\"{o}lder holonomies of the square $M=[0,1]^2$, with some topology and we show that a generic such foliation is not absolutely continuous, furthermore, the conditional measures defined by Rokhlin disintegration are Dirac measures on the leaves. This space of foliations is motivated by the foliations that appear in hyperbolic systems and partially hyperbolic systems.
\end{abstract}

\maketitle

\tableofcontents

\section{Introduction} 

\ \
\ \

In dynamical systems, an important aspect for the study of ergodicity is the regularity of the invariant foliations for the system. A continuous foliation with $C^r$ leaves is a partition $\mathcal{F}$ of a manifold $M$(with dimension $d\geq2$) into $C^r$ submanifolds of dimension $k$, for some $0<k<d$ and $1\leq r\leq\infty$, such that for every $p\in M$, there exists a continuous local chart 
$$\Phi:B_1^k\times B_1^{d-k}\to M, \ \ \ (B_1^m \mbox{ denotes the unit ball in } \R^m)$$ with $\Phi(0,0)=p$ and such that the restriction of every horizontal $B_1^k\times\{\eta\}$ is a $C^r$ embedding depending continuously on $\eta$ and whose image is contained in some $\mathcal{F}$-leaf. The image of such a chart $\Phi$ is a foliation box and $\Phi(B_1^k\times\{\eta\})$ are the corresponding local leaves. We say that a foliation $\F$ is absolutely  continuous if given any pair of smooth transversals to the foliation $\tau_1$ and $\tau_2$, the $\F$-holonomy $h_{\F}$ is absolutely continuous with respect to the induced Riemannian measures in the transversals $\lambda_{\tau_1}$ and $\lambda_{\tau_2}$. In section \ref{rok} we give more details about the different definitions of the absolute continuity of foliations.

\ \

In 1967, Anosov \cite{A67} prove that a transitive Anosov diffeomorphism of class $C^2$ that preserves the volume is ergodic (using the classic Hopf's argument). To prove this, an important step was to prove that the stable and unstable foliations are absolutely continuous. Recall that a diffeomorphism $f$ of class $C^r$ with $r\geq 1$ is an Anosov diffeomorphism if there exists a $Df$- invariant splitting of the tangent bundle $TM=E^s\oplus E^u$ and a Riemannian metric on $M$ such that the vector in $E^s$ are uniformly contracted by $Df$ and the vectors in $E^u$ are uniformly expanded.  After the work of Anosov, this result was generalized to Anosov diffeomorphisms of class $C^{1+\alpha}$ (a proof of this can be found in \cite{M87}). A few years later, Pugh and Shub in 1972 \cite{PS72} and Brin with Pesin in 1974 \cite{BP74} proved that if $f:M\to M$ is a partially hyperbolic diffeomorphism of class $C^{1+\alpha}$, then the stable and unstable foliations are absolutely continuous. Recall that a diffeomorphism $f$ of class $C^r$ with $r\geq 1$ is a partially hyperbolic diffeomorphism if there exists a $Df$- invariant splitting of the tangent bundle $TM=E^s\oplus E^c \oplus E^u$ and a Riemannian metric on $M$ such that the vector in $E^s$ are uniformly contracted by $Df$, the vectors in $E^u$ are uniformly expanded and we have an intermediate behaviour for the vectors in $E^c$. Besides, in the year 1976, Pesin \cite{P76} proved that for non-uniformly hyperbolic diffeomorphisms of class $C^{1+\alpha}$, the stable and unstable foliations are absolutely continuous. For more details of the definition of these diffeomorphisms, see \cite{BP01}.

\ \

With all of this, a natural question was considering a diffeomorphism of class $C^1$ and ask if the stable and unstable foliations are absolutely continuous, but Robinson and Young in 1980 \cite{RY80} constructed a $C^1$ Anosov diffeomorphism in $\T^2$ such that the stable and unstable foliations are non-absolutely continuous. 

\ \

In another direction, many authors have studied the absolute continuity of the central foliation of a partially hyperbolic diffeomorphism. Shub and Wilkinson in 2000 \cite{SW00} considered the automorphism of $\T^3$, given by $A_3 = \begin{pmatrix} A_2 & 0 \\ 0 & 1 \end{pmatrix}$, where $A_2=\begin{pmatrix} 2 & 1 \\ 1 & 1 \end{pmatrix}$, and they proved that arbitrarily close to $A_3$, there exists a $C^1$-open set $U\subset Diff^2_{\mu}(\T^3)$ such that for each $g\in U$, the central foliation $\F_{g}^c$ is not absolutely continuous. Furthermore, in the same year, Ruelle and Wilkinson \cite{RW01} proved that if $g\in U$, then the foliation $\F_{g}^c$ are absolutely singular, this mean that the conditional measures defined by the Rokhlin disintegration are atomic. Later in 2003, Baraviera and Bonatti \cite{BB03} considered a compact Riemannian manifold endowed with a $C^2$-volume form $\omega$, a $C^1$ Anosov flow $X:\R\times M\to M$ that preserves the volume $\omega$, and $f$ the time-one map of $X$ (that is a $C^1$ partially hyperbolic diffeomorphism), then for all $g$ inside an open set $C^1$-close to $f$, $\F_{g}^c$ and any leaf $L_c$ of $\F_{g}^c$, the set of points of $L_c$ having positive Lyapunov exponents has Lebesgue measure $0$ in $L_c$, and this implies that $\F_{g}^c$ is non-absolutely continuous with respect to Lebesgue for any $\omega$-preserving $g$ close to $f$ such that $$\int_M \log J_{g}^c(x)d\omega (x)>0.$$

In 2007, Hirayama and Pesin \cite{HP07} find sufficient conditions to obtain a non-absolutely continuous central foliation. In this case, they considered a $C^2$ partially hyperbolic diffeomorphism $f$ in a Riemannian compact manifold that preserves a smooth measure $\mu$ such that
\begin{enumerate}
\item
the central distribution $E^c$ is integrable to a foliation $\F^c$ with smooth compact leaves;
\item
$f$ has negative (positive) central exponents.
\end{enumerate}

Then the central foliation $\F^c$ is non-absolutely continuous. Furthermore, if $\mu$ is ergodic, then the conditional measures induced by $\mu$ on the leaves of $\F^c$ are atomic. A couple of years later, Saghin and Xia \cite{SX09} considered a linear automorphism $T_A:\T^n\to\T^n$ with splitting dominated $T\T^n=E^1\oplus E^2\oplus E^3$, $E^2$ uniformly expanding such that $J_2$ (the Jacobian of $T_A$ on $E_2$) is a simple eigenvalue of ${T_A}_{*}:H_{k_2}(\T^n,\R)\to H_{k_2}(\T^n,\R)$ and there is no other eigenvalue of absolute value of $J_2$, then there exist an open set of volume preserving diffeomorphisms $U$, $C^1$ arbitrarily close to $T_A$, such that for any $f\in U$, the weak unstable foliation of $f$, $\F_2$ is non-absolutely continuous.  Another result about pathological foliations is given by Gogolev in 2012 \cite{G12} where he showed that for a large set of volume preserving partially hyperbolic diffeomorphisms of the $\T^3$ with non-compact central leaves, the central foliation is non-absolutely continuous. Also, Viana and Yang in 2013 \cite{VY13} proved that for any small $C^1$-neighborhood $\mathcal{W}$ of the $C^k$ partially hyperbolic diffeomorphism $f_0=g_0\times id$  with $k>1$ in the space of volume preserving diffeomorphisms of $N=M\times \Su^1$, where $g_0$ is Anosov transitive diffeomorphism in a compact manifold $M$, 

\begin{enumerate}
\item
$\mathcal{W}_0=\{f\in\mathcal{W}:\lambda^c(f)\neq0\}$ is $C^1$-open and dense in $\mathcal{W}$, where $\lambda^c(f)$ is the integrated center Lyapunov exponent of $f$ relative to the Lebesgue measure;
\item
if $f\in\mathcal{W}$ and $\lambda^c(f)>0$ ($\lambda^c(f)<0$) then the center foliation and the center stable (unstable) foliation are not absolutely continuous;
\item
there exists a non-empty $C^1$-open set $\mathcal{W}_1\subset\{f\in\mathcal{W}_0:\lambda^c (f)>0\}$ such that the center unstable foliation of every $C^k$ diffeomorphism $g\in\mathcal{W}_1$ is absolutely continuous.
\end{enumerate}

Finally, the last result that we will mention was made by \'Avila, Viana and Wilkinson \cite{AVW15}, where they proved that for a volume-preserving perturbation $C^1$-close of the time-one map of the geodesic flow of a compact surface with negative curvature, the Liouville measure has Lebesgue disintegration along the center foliation, or the disintegration is necessarily atomic.

\ \

To sum up, it is important to mention that all of these results consider foliations which are given directly by the dynamics of some diffeomorphism and obtain that generically they are not absolutely continuous. In the literature, this fact is called "pathological foliations" or "Fubini nightmare", for example, in the work of Milnor \cite{M97}. However, Milnor mentioned that Yorke did a similar construction, based on tent maps. 

\ \

The goal of this paper is to show that in the absence of enough regularity, a generic foliation (not necessarily of dynamical origin) is non-absolutely continuous. For this purpose we are considering an abstract space of foliations, so let $M=[0,1]^2$ and $\mu$ the Lebesgue measure in $M$. Now let us define the space of foliations $\F_{(C,\beta)}$: Let $C>1$, $0\leq\alpha<\beta<1$ and consider the functions $f:M\to\R$ such that:

\begin{align}
&\quad \quad \quad f \mbox{ is } C^1 \mbox{ uniformly in the first variable,} & \\
&\quad \quad \quad f(0,y)=y,\mbox{ for all }y\in[0,1], & \\
&\quad \quad \quad f(x,0)=0,\mbox{ for all }x\in[0,1], & \\
&\quad \quad \quad f(x,1)=1,\mbox{ for all }x\in[0,1] \mbox{ and} & \\
&\quad \quad \quad f \mbox{ is } (C,\beta)\mbox{-bi-H\"{o}lder in the second variable}.
\end{align}

For the last condition, we mean that for all $x\in[0,1]$ and $y_1,y_2\in[0,1]$, with $y_1\neq y_2$, $$C^{-1}|y_2-y_1|^{1/\beta}\leq|f(x,y_2)-f(x,y_1)|\leq C|y_2-y_1|^\beta.$$ 
Denote $\F_{(C,\beta)}$ the set of these functions, and note that each $f\in\F_{(C,\beta)}$ represents a foliation of $M$, i.e., for each $f\in\F_{(C,\beta)}$, the graph of the function $f(\cdot,y):[0,1]\to[0,1]$ represent a leaf of this foliation, for every $y\in[0,1]$. This allows us to define a "metric between foliations": for $f,g\in\F_{(C,\beta)}$, define the metric 
$$d_{\alpha}(f,g):=\|f-g\|_{C^0}+\left\|\frac{\partial f}{\partial x}-\frac{\partial g}{\partial x}\right\|_{C^0}+\sup_{x\in[0,1]}\{\|h_{0,x}^f-h_{0,x}^g\|_{\alpha}\},\|h_{x,0}^f-h_{x,0}^g\|_{\alpha}\},$$
where 
$$\|f\|_{C^0}=\sup_{(x,y)\in M}\{|f(x,y)|\},\mbox{ \ \ }\|h\|_{\alpha}=\sup_{y_1,y_2\in[0,1]}\left\{\frac{|h(y_2)-h(y_1)|}{|y_2-y_1|^{\alpha}}\right\},$$
 and $h_{0,x}^f=f(x,\cdot)$, $h_{x,0}^{f}=f(x,\cdot)^{-1}$ are the holonomies of the foliation with leaves the graphs of $f(\cdot,y)$, between the vertical lines through $0$ and $x$. 
 
\begin{figure}[h]
\begin{center}
\includegraphics[width=6.06cm]{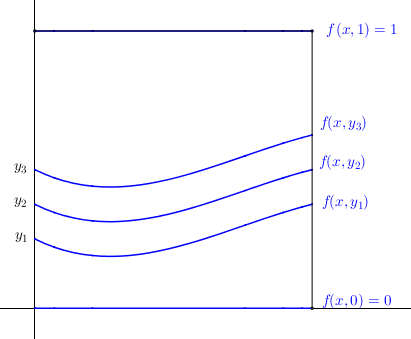}
\end{center}
\caption{$f\in\F_{(C,\beta)}$} \label{etiqueta}
\end{figure} 

\begin{rem} In general, for a $C^1$ Anosov diffeomorphism or partially hyperbolic diffeomorphism, the stable and unstable foliations are H\"{o}lder continuous with $C^1$ leaves, the same type as the foliations defined in this paper. The same occurs for the center foliations of $C^r$ partially hyperbolic diffeomorphisms, with $r\geq 1$.
\end{rem}

An important aspect is that given $f\in\F_{(C,\beta)}$, we can define a partition $\PP_{f}$ of $[0,1]^2$, where each element of the partition $\PP_f$ is the graph of $f(\cdot,y)$, for each $y\in[0,1]$, i.e., $$P\in\PP_f \mbox{ if and only if } P=\left\{f(x,y):x\in[0,1]\right\},$$ for some $y\in[0,1]$. Furthermore, this partition $\PP_f$ is measurable (in section \ref{rok} we give more details), so for almost every $P\in\PP_f$, there exist a probability measure $\mu_P$ (called conditional measure) supported in $P$. With all of this, we can enunciate the main theorem:

\begin{teo} \label{main} There exists a residual set $\mathcal{R}\subseteq\F_{(C,\beta)}$ such that for all $f\in\mathcal{R}$, the foliation $\PP_f$ is non-absolutely continuous. Furthermore, for $f\in\mathcal{R}$ and $\hat{\mu}$-a.e. $P\in\PP_f$, $\mu_P$ is mono-atomic, where $\mu_P$ is the conditional measure relative to $P$.
\end{teo}

\begin{rem} The set $\mathcal{R}$ is not open. The smooth foliations are dense in $\F_{(C,\beta)}$, and they are absolutely continuous.
\end{rem}

In section \ref{rok} we define the disintegration of a measure, we mentioned the Rokhlin's disintegration theorem, we define the different notions about the absolute continuity of foliations and discuss some properties that relate the absolute continuity of a foliation with the conditional measures defined by the Rokhlin's disintegration. In section \ref{spa} we show that the space of foliations is a complete metric space, so the countable intersetion of open and dense sets are dense. In section \ref{pro} we define sets $A_{n,m,I}$ that are open, then the sets $B_{m,I}=\bigcup_{n\in\N}A_{n,m,I}$ they are also open, and we show that the sets $B_{m,I}$ are dense, which is the main proposition in this paper. In section \ref{teo} we show the proof of the Theorem \ref{main} and some important remarks.

\ \
\ \

\section{Rokhlin's disintegration and absolute continuity of foliations}{\label{rok}}

\ \
\ \

In this section we are going to recall the classic result of Rokhlin. Let $M$ be a separable complete metric space, $\mu$ a Borel measure on $M$ and $\PP$ a partition of $M$. Denote by $\pi:M\to\PP$ the natural projection that associates to each point $x\in M$, the element $P(x)$ of the partition containing $x$. We say that $\mathcal{Q}\subset\PP$ is measurable if 
$$\pi^{-1}(\mathcal{Q})=\mbox{union of elements } P \mbox{ of } \PP \mbox{ that belongs to } \mathcal{Q}$$ 
is a measurable subset of $M$. It is easy to see that the family $\hat{\mathcal{B}}$ of the measurable sets is a $\sigma$-\'algebra in $\PP$. With this, define the quotient measure $$\hat{\mu}(\mathcal{Q}):=\mu(\pi^{-1}(\mathcal{Q})), \mbox{ for each }\mathcal{Q}\in\hat{\mathcal{B}}.$$

\begin{defi}
We say that $\mu$ has a disintegration relative to a partition $\PP$ if there exists a family $\{\mu_P:P\in\PP\}$ of probabilities in $M$ such that for all measurable set $E\subset M$:
\begin{enumerate}
\item
$\mu_P(P)=1$, for $\hat{\mu}$-a.e. $P\in\PP$.
\item
The map $P\to\mu_P(E)$ is measurable.
\item
$\mu(E)=\ds\int\mu_P(E)d\hat{\mu}(P)$.
\end{enumerate}
Such probabilities $\mu_P$ are called conditional measures (probabilities) of $\mu$ relative to $\PP$.
\end{defi}

\begin{prop} \cite{Ro67} Suppose that the $\sigma$-algebra $\mathcal{B}$ admits some countable generator. If $\{\mu_P:P\in\PP\}$ and $\{\nu_P:P\in\PP\}$ are two disintegrations for $\mu$ with respect to $\PP$, then $\mu_P=\nu_P$, for $\hat{\mu}$-a.e. $P\in\PP$.
\end{prop}

\begin{defi} We say that $\PP$ is a measurable partition if there exists a measurable set $M_0\subset M$ with full measure such that, restrict to $M_0$, 
$$\PP=\bigvee_{n=1}^\infty \PP_n,$$ 
for some increasing sequence $\PP_1 \prec\PP_2\prec\ldots\prec\PP_n \prec\ldots$ of countable partitions. Remember that $\PP_i\prec\PP_{i+1}$ means that every element of $\PP_{i+1}$ is contained in some element of $\PP_i$. The elements $P\in\PP$ are non-empty intersections of the form $P=\cap_{n=1}^\infty P_n$, where $P_n\in\PP_n$, for all $n\in\N$.
\end{defi}

\begin{teo}[\textbf{Rokhlin's Disintegration}]
Suppose that $M$ is a separable complete space, $\mu$ a probability measure and $\PP$ is a measurable partition, then $\mu$ has a disintegration with respect to $\PP$.
\end{teo}

For more details, see \cite{Vi16}.

\begin{obs} In the proof of Rokhlin's Disintegration Theorem, we have the explicit definition of the conditional measures: For almost everyl $P\in\PP$ and $P_n\in\PP_n$ such that $P=\cap_{n=1}^\infty P_n$, then $$\mu_P(A)=\lim_{n\to\infty}\frac{\mu(P_n\cap A)}{\mu(P_n)},$$ for all $A\subset M$ measurable.
\end{obs}

Now we are going to discuss some definitions and results of the absolute continuity of foliations.

\begin{defi} We say that a foliation $\F$ is absolutely  continuous if given any pair of smooth transversals to the foliation $\tau_1$ and $\tau_2$, the $\F$-holonomy $h_{\F}$ is absolutely continuous with respect to the induced Riemannian measures in the transversals $\lambda_{\tau_1}$ and $\lambda_{\tau_2}$, meaning that if $A\subset\tau_1$ and $\lambda_{\tau_1}(A)=0$, then $h_{\F*}\lambda_{\tau_2}(A)=0$.
\end{defi}

\begin{defi} We say that a foliation $\mathcal{F}$ is:
\begin{enumerate}
\item
\textit{Leafwise absolutely continuous I} if for any zero-set $A$ and for $m$-a.e. $p$, $\lambda_{\mathcal{F}_{p}}(A)=0$.
\item
\textit{Leafwise absolutely continuous II} if for any measurable set $A$ such that $\lambda_{\mathcal{F}_{p}}(A)=0$ for $m$-a.e. $p$, then $A$ is a zero-set for $m$.
\item
\textit{Leafwise absolutely continuous III} if $\mathcal{F}$ is both leafwise absolutely continuous I and leafwise absolutely continuous II
\end{enumerate}
\end{defi}

In terms of disintegration, we can say the following:

\begin{lemma}
\begin{enumerate}
\item
$\mathcal{F}$ is leafwise absolutely continuous I if and only if, for $m$-a.e. $p$, the measure $\lambda_{\mathcal{F}_{p}}$ is absolutely continuous with respect to the disintegration $m_p$.
\item
$\mathcal{F}$ is leafwise absolutely continuous II if and only if, for $m$-a.e. $p$, the disintegration $m_p$ is absolutely continuous with respect to $\lambda_{\mathcal{F}_{p}}$.
\item
$\mathcal{F}$ is leafwise absolutely continuous III if and only if, for $m$.a.e. $p$, the disintegration $m_p$ is equivalent to $\lambda_{\mathcal{F}_p}$.
\end{enumerate}
\end{lemma}

Another important lemma is the following:

\begin{lemma} $\mathcal{F}$ is leafwise absolutely continuous III if there exists a
transverse local foliation $\mathcal{T}$ to $\mathcal{F}$ such that $\mathcal{T}$ is absolutely continuous, and such that the  $\mathcal{F}$-holonomy between almost every pair of $\mathcal{T}$-leaves is absolutely continuous.
\end{lemma}

This lemma implies an important corollary, that assures us the non-absolutely continuity of generic foliations in the result of this paper.

\begin{cor} If $\mathcal{F}$ is absolutely continuous, then $\mathcal{F}$ is leafwise absolutely continuous III.
\end{cor}

For more details of these definitions and results see \cite{PVW07}.

\newpage

\section{The space of foliations $\F_{(C,\beta)}$}{\label{spa}}

\ \
\ \

\begin{obs}
\begin{enumerate}
\item
If $\beta'<\beta$, then $\F_{(C,\beta)}\subseteq\F_{(C,\beta')}$.
\item
If $C'<C$, then $\F_{(C',\beta)}\subseteq\F_{(C,\beta)}$.
\end{enumerate}
\end{obs}

The first thing is to have a convenient topology for the space $\F_{(C,\beta)}$.

\begin{lemma} $(\F_{(C,\beta)},d_{\alpha}(\cdot,\cdot))$ is a complete metric space.
\end{lemma}

\begin{proof}
In first instance, we are going to prove that $(\F_{(C,\alpha)},d_{\alpha}(\cdot,\cdot))$ is a complete metric space, and then prove that $\F_{(C,\beta)}$ is a closed subspace of $\F_{(C,\alpha)}$. Let $\{f_n\}\subseteq\F_{(C,\alpha)}$ be a Cauchy sequence. So, for all $\varepsilon>0$, there exists $n_0\in\N$ such that for all $n,m\geq n_0$, 
$$d_{\alpha}(f_n,f_m)=\|f_n-f_m\|_{C^0}+\left\|\frac{\partial f_n}{\partial x}-\frac{\partial f_m}{\partial x}\right\|_{C^0}+\sup_{x\in[0,1]}\{\|h_{0,x}^{f_n}-h_{0,x}^{f_m}\|_{\alpha}\},\|h_{x,0}^{f_n}-h_{x,0}^{f_m}\|_{\alpha}\}<\varepsilon.$$ 
Since $\left\{f_n\right\},\left\{\ds\frac{\partial f_n}{\partial x}\right\}$ are Cauchy sequences, this implies that $f_n\to f$ and $\ds\frac{\partial f_n}{\partial x}\to\frac{\partial f}{\partial x}$ in the $C^0$ topology, for some $f$. Now, considering the fact that the space of functions $(C,\alpha)$-H\"{o}lder in $[0,1]$ is a Banach space, then we have that $h_{0,x}^{f_n}=f_n(x,\cdot)\to f(x,\cdot)$ in the $C^{\alpha}$ topology and $f(x,\cdot)$ is $(C,\alpha)$-H\"{o}lder in the second variable. For $h_{x,0}^{f_n}$ we have something similar. Thus, $f_n\to f$ in the $d_{\alpha}$ topology and $f\in\F_{(C,\alpha)}$, so, $\F_{(C,\alpha)}$ is a complete metric space. The final step is prove that $\F_{(C,\beta)}$ is closed in $\F_{(C,\alpha)}$ in the $d_\alpha$-topology. For this, take $\{f_n\}\subset\F_{(C,\beta)}$ such that $f_n\to f$ in the $d_\alpha$-topology. Now, the only thing to check is that $f$ is $(C,\beta)$-bi-H\"{o}lder in the second variable. If $f_n\to f$ in the $d_\alpha$-topology, then for all $\varepsilon>0$, there exists $N\in\N$ such that for $n\geq N$, $$d_{\alpha}(f,f_n)=\|f-f_n\|_{C^0}+\left\|\frac{\partial f}{\partial x}-\frac{\partial f_n}{\partial x}\right\|_{C^0}+\sup_{x\in[0,1]}\{\|h_{0,x}^{f}-h_{0,x}^{f_n}\|_{\alpha}\},\|h_{x,0}^{f}-h_{x,0}^{f_n}\|_{\alpha}\}<\varepsilon.$$ In the last term, we have that $\|h_{0,x}^f-h_{0,x}^{f_n}\|_{\alpha}<\varepsilon$, and this implies that $$|f(x,y_2)-f_n(x,y_2)-(f(x,y_1)-f_n(x,y_1))|<\varepsilon|y_2-y_1|^\alpha,$$ for all $y_1,y_2\in[0,1]$, with $y_1\neq y_2$. Then,

\begin{align*}
|f(x,y_2)-f(x,y_1)|&\leq|f(x,y_2)-f_n(x,y_2)-(f(x,y_1)-f_n(x,y_1))| \\
&\mbox{ \ \ }+|f_n(x,y_2)-f_n(y_1)| \\
&<\varepsilon|y_2-y_1|^\alpha+C|y_2-y_1|^\beta.
\end{align*}

Take $\varepsilon>0$, so $|f(x,y_2)-f(x,y_1)|\leq C|y_2-y_1|^\beta$. Is very similar to see that $|f(x,y_2)-f(x,y_1)|\geq C^{-1}|y_2-y_1|^{1/\beta}$. Thus, $f\in\F_{(C,\beta)}$.
\end{proof}

Since a complete metric space is a Baire space, we have that $\F_{(C,\beta)}$ is a Baire space, so the countable intersection of open and dense sets are dense, and such a set is called a residual.

\ \ 
\ \

\section{Main Proposition}{\label{pro}}

\ \
\ \

Observe that if we consider $f\in\F_{(C,\beta)}$ and $n\in\N$, then we can define (finite) partitions $\PP_{f,n}$ for $M$: $P_i\in \PP_{f,n}$ if and only if $$P_i=\left\{(x,y)\in M:y\in\left[f\left(x,\ds\frac{i}{2^n}\right),f\left(x,\ds\frac{i+1}{2^n}\right)\right)\right\},$$ where $i=0,\ldots,2^n-1$. So, for each $f\in\F_{(C,\beta)}$, we can define measurable partitions for $M$:
 $$\PP_{f}=\bigvee_{n=1}^\infty\PP_{f,n}.$$

Now, let $m\in\N$, $I=[b_1,b_2]\subset[0,1]$ with $b_1,b_2\in\Q$ and $\tilde{I}=I\times[0,1]$. Define the sets $A_{n,m,I}$: 
$$A_{n,m,I}=\left\{f\in\F_{(C,\beta)}:\frac{\mu(P\cap \tilde{I})}{\mu(P)}<\frac{1}{m}\vee\frac{\mu(P\cap \tilde{I})}{\mu(P)}>1-\frac{1}{m},\mbox{ for all }P\in \PP_{f,n}\right\}.$$ 
Clearly, $A_{n,m,I}$ are open, so we can define the sets $B_{m,I}=\bigcup_{n\in\N}A_{n,m,I}$.

\begin{prop} The sets $B_{m,I}$ are open and dense.
\end{prop}

\begin{proof} It is clear that $B_{m,I} $ are open sets. Now we are going to prove the proposition in several steps. Let $f\in\F_{(C,\beta)}$ and $\xi>0$.

\begin{lemma} \label{lema2} There exists $f_2\in\F_{(C,\beta)}$ such that $f_2$ is $C^2$ in the second variable and $d_\alpha(f_2,f)<\xi/3$.
\end{lemma}

\begin{proof} The first thing to do is extend $f$ in the second variable; for this, take any $r>0$ such that $$r\leq\min\left\{\left(\ds\frac{\xi}{12C}\right)^{1/\beta},\left(\ds\frac{\xi}{24C}\right)^{1/(\beta-\alpha)},\ds\frac{\xi}{24}\cdot\left(\ds\frac{\xi}{12C^\beta}\right)^{\alpha/(\beta-\alpha)}\right\} $$ and $C_r\left(\ds\frac{\partial f}{\partial x}\right)\leq\ds\xi/12$, where $C_r(h)=\ds\sup_{|x-y|<r}\{|h(x)-h(y)|\}$, and define the function $f_1:[0,1]\times(-r,1+r)\to[0,1]$ such that 

$$f_1(x,y)= \left\{ \begin{array}{ll}
-f(x,y), & \mbox{ if } y\in(-r,0] \\
    \\
    f(x,y), & \mbox{ if } y\in[0,1] \\
    \\
   1-f(x,1-y), & \mbox{ if } y\in[1,1+r).
             \end{array}
   \right.$$
 
With this function, we can define a function $C^2$ in the second variable through the convolution, such that is close to $f$ in the $d_\alpha$-topology: let $\phi_r:\R\to\R^+$ a bump function such that $\phi_r(x)=0$, for all $x\notin(-r,r)$ and $\int_\R\phi_r(t)dt=1$, so define $$f_2(x,y)=\int_{-r}^{r}f_1(x,y-t)\phi_r(t)dt.$$ 

In fact, $f_2\in\F_{(C,\beta)}$: if $y_1,y_2\in[0,1]$ with $y_2>y_1$, then 

\begin{align*}
|f_2(x,y_2)-f_2(x,y_1)|&=\left|\int_{-r}^r(f_1(x,y_2-t)-f_1(x,y_1-t))\phi_r(t)dt\right| \\
&\leq\int_{-r}^r C|y_2-y_1|^\beta\phi_r(t)dt \\
&=C|y_2-y_1|^\beta.
\end{align*}

It is the same to see that $|f_2(x,y_2)-f_2(x,y_1)|\geq C^{-1}|y_2-y_1|^{1/\beta}$, and the other conditions are trivial, thus $f_2\in\F_{(C,\beta)}$. Now, we have to check that $d_\alpha(f_2,f)<\xi/3$.

\begin{enumerate}
\item
If $(x,y)\in M$, then 

\begin{align*}
|f_2(x,y)-f(x,y)|&=\left|\int_{-r}^r(f_1(x,y-t)-f_1(x,y))\phi_r(t)dt\right| \\
&\leq\left|\int_{-r}^r C|t|^\beta\phi_r(t)dt\right| \\
&\leq Cr^\beta \\
&\leq C\cdot\left(\left(\frac{\xi}{12C}\right)^{1/\beta}\right)^\beta \\
&=\frac{\xi}{12}
\end{align*}

So, $\|f_2-f\|_{C^0}\leq\xi/12$.

\ \

\item
If $(x,y)\in M$, then 

\begin{align*}
\left|\frac{\partial f_2}{\partial x}(x,y)-\frac{\partial f}{\partial x}(x,y)\right|&=\left|\int_{-r}^r\left(\frac{\partial f_1}{\partial x}(x,y-t)-\frac{\partial f_1}{\partial x}(x,y)\right)\phi_r(t)dt\right| \\
&\leq C_{r}\left(\frac{\partial f}{\partial x}\right) \\
&\leq\frac{\xi}{12}.
\end{align*}

So, $\left\|\ds\frac{\partial f_2}{\partial x}-\ds\frac{\partial f}{\partial x}\right\|_{C^0}\leq \xi/12$.

\ \

\item
Notice that $h_{0,x}^f(y)=f(x,y)$, so in a first case, if $y_1,y_2\in[0,1]$ with $y_2>y_1$ such that $|y_2-y_1|<\left(\ds\frac{\xi}{24C}\right)^{1/(\beta-\alpha)}$, then 

\begin{align*}
|h_{0,x}^{f_2}(y_2)-h_{0,x}^{f}(y_2)-h_{0,x}^{f_2}(y_1)+h_{0,x}^{f}(y_1)|&=|f_2(x,y_2)-f(x,y_2)-f_2(x,y_1)+f(x,y_1)| \\
&=\left|\int_{-r}^r\bigl[(f_1(x,y_2-t)-f_1(x,y_1-t))\right. \\
&\mbox{ \ \ \ }\Bigl.-(f_1(x,y_2)-f_1(x,y_1))\bigr]\phi_r(t)dt\Bigr| \\
&\leq\int_{-r}^r 2C|y_2-y_1|^\beta \phi_r(t)dt \\
&=2C|y_2-y_1|^{\beta-\alpha}|y_2-y_1|^\alpha \\
&<2C\left(\left(\frac{\xi}{24C}\right)^{1/(\beta-\alpha)}\right)^{\beta-\alpha}\cdot|y_2-y_1|^\alpha \\
&=\frac{\xi}{12C}|y_2-y_1|^\alpha.
\end{align*}

Now, if $|y_2-y_1|\geq \left(\ds\frac{\xi}{24C}\right)^{1/(\beta-\alpha)}$, then 

\begin{align*}
|h_{0,x}^{f_2}(y_2)-h_{0,x}^{f}(y_2)-h_{0,x}^{f_2}(y_1)+h_{0,x}^{f}(y_1)|&=|f_2(x,y_2)-f(x,y_2)-f_2(x,y_1)+f(x,y_1)| \\
&=\left|\int_{-r}^r\bigl[(f_1(x,y_2-t)-f_1(x,y_2))\right. \\
&\mbox{ \ \ \ }\Bigl.-(f_1(x,y_1-t)-f_1(x,y_1))\bigr]\phi_r(t)dt\Bigr|\\
&\leq\left|\int_{-r}^r(C|t|^\beta+C|t|^\beta)\phi_r(t)dt\right| \\
&\leq 2Cr^\beta \\
&\leq 2C\left(\frac{\xi}{24C}\right)^{\beta/(\beta-\alpha)} \\
&=2C\left(\frac{\xi}{24C}\right)\cdot\left(\frac{\xi}{24C}\right)^{\alpha/(\beta-\alpha)} \\
&\leq\frac{\xi}{12}|y_2-y_1|^\alpha.
\end{align*}

Therefore, $\|h_{0,x}^{f_2}-h_{0,x}^f\|_{\alpha}\leq\ds\frac{\xi}{12}$.

\item
Define $h_{x,0}^{f_2}=f_2(x,\cdot)^{-1}(y):=y^{f_2}$. Now, if $0<y_2-y_1<\left(\ds\frac{\xi}{12C^\beta}\right)^{1/(\beta-\alpha)}$, then 

\begin{align*}
|(h_{x,0}^{f_2}-h_{x,0}^f)(y_2-y_1)|&=|y_2^{f_2}-y_2^f-y_1^{f_2}+y_1^f| \\
&=|y_2^{f_2}-y_1^{f_2}-(y_2^f-y_1^f)| \\
&<|y_2^{f_2}-y_1^{f_2}| \\
&\leq C^\beta|y_2-y_1|^\beta \\
&=C^\beta|y_2-y_1|^{\beta-\alpha}\cdot|y_2-y_1|^\alpha \\
&<\frac{\xi}{12}|y_2-y_1|^\alpha.
\end{align*}

Now, if $|y_2-y_1|\geq\left(\ds\frac{\xi}{12C^\beta}\right)^{1/(\beta-\alpha)}$, then

\begin{align*}
|(h_{x,0}^{f_2}-h_{x,0}^f)(y_2-y_1)|&=|y_2^{f_2}-y_2^f-y_1^{f_2}+y_1^f| \\
&<|y_2^{f_2}-y_2^f| \\
&<2r \\
&<2\cdot\frac{\xi}{24}\cdot\left(\frac{\xi}{12C^\beta}\right)^{\alpha/(\beta-\alpha)} \\
&\leq\frac{\xi}{12}|y_2-y_1|^\alpha.
\end{align*}

So, $\|h_{x,0}^{f_2}-h_{x,0}^f\|_\alpha\leq\ds\xi/12$. To sum up,

\end{enumerate}

\begin{align*}
d_\alpha(f_2,f)&=\|f_2-f\|_{C^0}+\left\|\frac{\partial f_2}{\partial x}-\frac{\partial f}{\partial x}\right\|_{C^0}+\sup_{x\in[0,1]}\{\|h_{0,x}^{f_2}-h_{0,x}^f\|_\alpha,\|h_{x,0}^{f_2}-h_{x,0}^f\|_\alpha\} \\
&\leq Cr^\beta+C_r\left(\frac{\partial f}{\partial x}\right)+\frac{\xi}{12}+\frac{\xi}{12} \\
&\leq \frac{\xi}{3}.
\end{align*}
 
\end{proof}

Until now we have defined a function $f_2$, which is $C^2$ in the second variable, in particular, $f_2$ is bi-Lipschitz in the second variable, i.e., there exists a constant $L_2>0$ such that $L_2^{-1}|y_2-y_1|\leq|f_2(x,y_2)-f_2(x,y_1)|\leq L_2|y_2-y_1|$, for all $y_1,y_2\in[0,1]$ and $x\in[0,1]$.

\begin{lemma} \label{lema3} For any constant $0<\varepsilon\leq\min\left\{\ds\frac{\xi}{12},\frac{\xi}{12\left\|\ds\frac{\partial f_2}{\partial x}\right\|_{C^0}},\frac{\xi}{12(1+C)},\left(\frac{\xi}{24C^\beta}\right)^{1/(\beta-\alpha)}\right\}$, there exists a function $f_3$ such that $f_3\in\F_{(C_\varepsilon,\beta)}$ for some $C_\varepsilon<C$ and $d_\alpha(f_3,f_2)<\xi/3$.
\end{lemma}

\begin{proof}

Let $\varepsilon>0$, $C_\varepsilon=\max\{C(1-\varepsilon)+\varepsilon, C/(1-\varepsilon+C\varepsilon)\}<C$ and define $f_3$ as an interpolation between $f_2$ and the identity in the second variable: $$f_3(x,y):=(1-\varepsilon)f_2(x,y)+\varepsilon y.$$ In fact, $f_3\in\F_{(C_\varepsilon,\beta)}$: if $y_1,y_2\in[0,1]$ with $y_2>y_1$, then

\begin{align*}
|f_3(x,y_2)-f_3(x,y_1)|&=|(1-\varepsilon)(f_2(x,y_2)-f_2(x,y_1))+\varepsilon(y_2-y_1)| \\
&\leq C(1-\varepsilon)|y_2-y_1|^\beta+\varepsilon|y_2-y_1|^{1-\beta}|y_2-y_1|^\beta \\
&\leq(C(1-\varepsilon)+\varepsilon)|y_2-y_1|^\beta \\
&\leq C_\varepsilon|y_2-y_1|^\beta
\end{align*}

In a similar way, we can see that $|f_3(x,y_2)-f_3(x,y_1)|\geq C_\varepsilon^{-1}|y_2-y_1|^{1/\beta}$:

\begin{align*}
|f_3(x,y_2)-f_3(x,y_1)|&=|(1-\varepsilon)(f_2(x,y_2)-f_2(x,y_1))+\varepsilon(y_2-y_1)| \\
&\geq \left(\frac{1-\varepsilon}{C}+\varepsilon\right)|y_2-y_1|^{1/\beta} \\
&=\frac{1}{\frac{C}{1-\varepsilon+C\varepsilon}}|y_2-y_1|^{1/\beta}\\
&\geq C_{\varepsilon}^{-1}|y_2-y_1|^{1/\beta}. 
\end{align*}

The other conditions are trivial, so, $f_3\in\F_{(C_\varepsilon,\beta)}$. Now, we have to check that $d_\alpha(f_3,f_2)<\xi/3$. In first instance, if $x,y\in[0,1]$, then

$$|f_3(x,y)-f_2(x,y)|=\varepsilon|y-f_2(x,y)|<\varepsilon\leq\frac{\xi}{12},$$ and $$\left|\frac{\partial f_3}{\partial x}(x,y)-\frac{\partial f_2}{\partial x}(x,y)\right|=\left|(1-\varepsilon)\frac{\partial f_2}{\partial x}(x,y)-\frac{\partial f_2}{\partial x}(x,y)\right|\leq\varepsilon \left\|\frac{\partial f_2}{\partial x}\right\|_{C^0}\leq\frac{\xi}{12}.$$ 

For the next condition, notice that $h_{0,x}^{f_3}(y)=f_3(x,y)$, so if $x,y_1,y_2\in[0,1]$ with $y_2>y_1$, then 

\begin{align*}
|h_{0,x}^{f_3}(y_2)-h_{0,x}^f(y_2)-h_{0,x}^{f_3}(x,y_1)+h_{0,x}^f(x,y_1)|&=|f_3(x,y_2)-f_2(x,y_2)-f_3(x,y_1)+f_2(x,y_1)| \\
&=|\varepsilon(y_2-y_1)-\varepsilon(f_2(x,y_2)-f_2(x,y_1))| \\
&\leq\varepsilon(|y_2-y_1|+|f_2(x,y_2)-f_2(x,y_1)|) \\
&<\varepsilon(|y_2-y_1|^\alpha+C|y_2-y_1|^\beta) \\
&<\varepsilon(|y_2-y_1|^\alpha+C|y_2-y_1|^\alpha) \\
&=\varepsilon(1+C)|y_2-y_1|^\alpha.
\end{align*}

So, $\|h_{0,x}^{f_3}-h_{0,x}^{f_2}\|_\alpha\leq\varepsilon(1+C)\leq\xi/12$. Now, define $h_{x,0}^{f_3}(y)=f_{\varepsilon}(x,\cdot)^{-1}(y):=y^{f_3}$. In a first case, if $|y_2-y_1|<(\xi/24C^\beta)^{1/(\beta-\alpha)}$, then, 

\begin{align*}
|h_{x,0}^{f_3}(y_2)-h_{x,0}^{f_2}(y_2)-h_{x,0}^{f_3}(y_1)+h_{x,0}^{f_2}(y_1)|&=|y_2^{f_3}-y_2^{f_2}-y_1^{f_3}+y_1^{f_2}| \\
&=|y_2^{f_3}-y_1^{f_3}-(y_2^{f_2}-y_1^{f_2})| \\
&<|y_2^{f_3}-y_1^{f_3}|+|y_2^{f_2}-y_1^{f_2}| \\
&\leq C_{\varepsilon}^\beta|y_2-y_1|^\beta +C^\beta|y_2-y_1|^\beta\\
&\leq 2C^\beta|y_2-y_1|^{\beta-\alpha}\cdot|y_2-y_1|^\alpha \\
&< 2C^\beta\cdot\left(\left(\frac{\xi}{24C^\beta}\right)^{1/(\beta-\alpha)}\right) ^{\beta-\alpha}\cdot|y_2-y_1|^\alpha \\
&=\frac{\xi}{12}|y_2-y_1|^\alpha.
\end{align*}

Now, if $|y_2-y_1|\geq(\xi/24C^\beta)^{1/(\beta-\alpha)}$, notice that the equality $f_3(x,y_i^{f_3})=f_2(x,y_i^{f_2})$ implies that $$C^{-1}|y_i^{f_3}-y_i^{f_2}|^{1/\beta}\leq|f_2(x,y_i^{f_3})-f_2(x,y_i^{f_2})|=\varepsilon|f_2(x,y_i^{f_3})-y_i^{f_3})|\leq\varepsilon,$$ for $i=1,2$. Then 

\begin{align*}
|(h_{x,0}^{f_3}-h_{x,0}^{f_2})(y_2-y_1)|&=|y_2^{f_3}-y_2^{f_2}-(y_1^{f_3}-y_1^{f_2})| \\
&<|y_2^{f_3}-y_2^{f_2}|+|y_1^{f_3}-y_1^{f_2}| \\
&\leq 2C^{\beta}\varepsilon^\beta \\
&\leq2C^\beta\cdot\left(\frac{\xi}{24C^\beta}\right)^{\beta/(\beta-\alpha)} \\
&=2C^\beta\cdot\frac{\xi}{24C^\beta}\cdot\left(\frac{\xi}{24C^\beta}\right)^{\alpha/(\beta-\alpha)} \\
&\leq\frac{\xi}{12}|y_2-y_1|^\alpha.
\end{align*}

So, $\|h_{x,0}^{f_3}-h_{x,0}^{f_2}\|_\alpha\leq\xi/12$. To sum up, all of this implies that 

$$d_\alpha(f_3,f_2)\leq \|f_3-f_2\|_{C^0}+\left\|\frac{\partial f_3}{\partial x}-\frac{\partial f_2}{\partial x}\right\|_{C^0}+\|h_{0,x}^{f_3}-h_{0,x}^{f_2}\|_{\alpha}+\|h_{x,0}^{f_3}-h_{x,0}^{f_2}\|_{\alpha} <\frac{\xi}{3}.$$
\end{proof}

Note that $f_3$ is also a $C^2$ function in the second variable, in particular, $f_3$ is bi-Lipschitz in the second variable, this means that there exists a constant $L_3>0$ such that $$L_3^{-1}|y_2-y_1|\leq|f_3(x,y_2)-f_3(x,y_2)|\leq L_3|y_2-y_1|,$$ for all $y_1,y_2\in[0,1]$ and $x\in[0,1]$.

Now we are able to define perturbations $\tilde{f}=\tilde{f}(\delta_1,\delta_2,n)$ of $f_3$: for $\delta_1,\delta_2>0$ small and $n\in\N$ large enough, let
$$a(x)=\frac{1}{\gamma}\int_0^x\exp\left(\frac{1}{(2t-1)^2-1}\right)dt,$$ 
where $\gamma=\ds\int_0^1\exp\left(\frac{1}{(2t-1)^2-1}\right)dt$, then $a$ is $C^\infty$, $a(0)=0$ and $a(1)=1$. Now, let $I=[b_1,b_2]$, so we can define a $C^1$ function $\tilde{a}(x)$ such that $\tilde{a}(0)=1/2$, $\tilde{a}_{(\delta_1,b_1-\delta_1)}=\tilde{a}|_{(b_2+\delta_1,1)}=\delta_2$ and $\tilde{a}|_{(b_1,b_2)}=1-\delta_2$.

$$\tilde{a}(x)= \left\{ \begin{array}{ll}
\left(\ds\frac{1}{2}-\delta_2\right)a\left(\ds\frac{x}{\delta_1}\right)+\ds\frac{1}{2}, & \mbox{ if } x\in[0,\delta_1) \\
    \\
    \delta_2, & \mbox{ if } x\in(\delta_1,b_1-\delta_1) \\
    \\
   (1-2\delta_2)a\left(-\ds\frac{x}{\delta_1}+\frac{b_1}{\delta_1}\right)+\delta_2, & \mbox{ if }x\in(b_1-\delta_1,b_1) \\
    \\
    1-\delta_2, & \mbox{ if }x\in(b_1,b_2) \\
    \\
    (1-2\delta_2)a\left(\ds\frac{x}{\delta_1}-\frac{b_2}{\delta_1}\right)+\delta_2, & \mbox{ if }x\in(b_2,b_2+\delta_1) \\
    \\
    \delta_2, & \mbox{ if } x\in(b_2+\delta_1,1]
             \end{array}
   \right.$$
   
For the definition of $\tilde{a}$, $\tilde{a}(x)\in[\delta_2,1-\delta_2]$ and $|\tilde{a}'(x)|\leq\ds\frac{e^{-1}(1-2\delta_2)}{\delta_1 \gamma}<\frac{2(1-2\delta_2)}{\delta_1}$. 

\begin{figure}[htbp]
\centering
\subfigure[Graph of $a$]{\includegraphics[width=5cm]{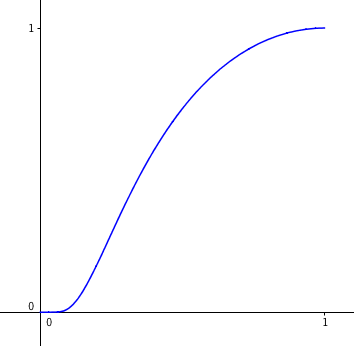}}\hspace{20mm}
\subfigure[Graph of $\tilde{a}$]{\includegraphics[width=5cm]{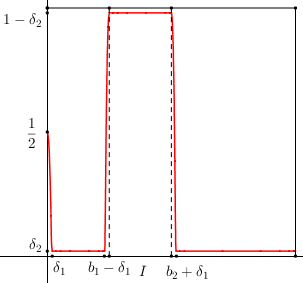}}
\end{figure}

Now define the perturbation for all $a\in\{0,\ldots,2^{n-1}-1\}$. The central curve of $\tilde{f}$ in $\left[\ds\frac{2a}{2^n},\ds\frac{2a+2}{2^n}\right]$ is: $$\tilde{f}\left(x,\frac{2a+1}{2^n}\right)=(1-\tilde{a}(x))f_3\left(x,\frac{2a}{2^n}\right)+\tilde{a}(x)f_3\left(x,\frac{2a+2}{2^n}\right).$$

\begin{figure}[htbp]
\centering
\subfigure[Graph of $f_3$]{\includegraphics[width=6.53cm]{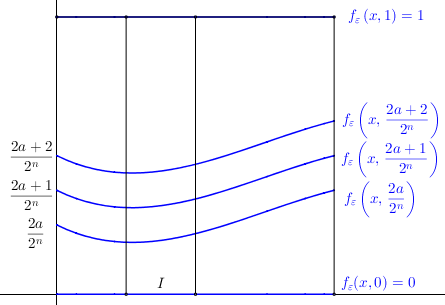}}\hspace{20mm}
\subfigure[Graph of $\tilde{f}$]{\includegraphics[width=6.53cm]{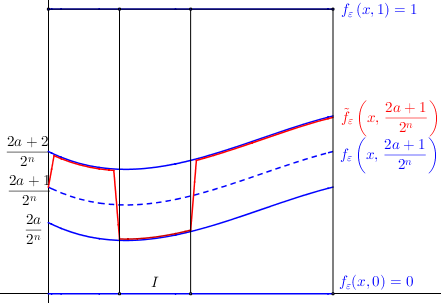}}
\end{figure}

\newpage

With this, we can define $\tilde{f}$ entirely as an interpolation: for $y\in\left[\ds\frac{2a}{2^n},\frac{2a+1}{2^n}\right]$, $$\tilde{f}(x,y)= \tilde{f}\left(x,(1-t)\ds\frac{2a}{2^n}+t\cdot \frac{2a+1}{2^n}\right):=(1-t)\tilde{f}\left(x,\ds\frac{2a}{2^n}\right)+t\tilde{f}\left(x,\ds\frac{2a+1}{2^n}\right)$$ and for $y\in\left[\ds\frac{2a+1}{2^n},\ds\frac{2a+2}{2^n}\right]$, $$\tilde{f}(x,y)=\tilde{f}\left(x,(2-t)\ds\frac{2a+1}{2^n}+(t-1)\frac{2a+2}{2^n}\right):=(2-t)\tilde{f}\left(x,\ds\frac{2a+1}{2^n}\right)+(t-1)\tilde{f}\left(x,\ds\frac{2a+2}{2^n}\right)$$

Then, for all $a\in\{0,\ldots,2^n-1\}$ and $y\in\left[\ds\frac{2a}{2^n},\ds\frac{2a+2}{2^n}\right]$, 

$$\tilde{f}\left(x,\frac{2a+t}{2^n}\right)= \left\{ \begin{array}{ll}
     (1-t\tilde{a}(x))f_3\left(x,\ds\frac{2a}{2^n}\right)+t\tilde{a}(x)f_3\left(x,\ds\frac{2a+2}{2^n}\right), & t\in[0,1] \\
     \\ (1-\tilde{a}(x))(2-t)f_3\left(x,\ds\frac{2a}{2^n}\right)+(t-1+\tilde{a}(x)(2-t))f_3\left(x,\ds\frac{2a+2}{2^n}\right), & t\in[1,2]
             \end{array}
   \right.$$

\begin{lemma} \label{lema4} If $\delta_1,\delta_2,n$ satisfy the following conditions:
\begin{enumerate}
\item
$\ds\frac{L_3(1-\delta_2)}{2^{n(1-\beta)-3}}+C_\varepsilon\leq C$;
\item
$C_\varepsilon\leq2\cdot\ds\frac{2^{n(1-1/\beta)}}{L_3^{-1}\delta_2}$;
\item
$\ds\frac{2^{n(1-1/\beta)+2/\beta}}{L_3^{-1}\delta_2}\leq C$;
\item
$\max\left\{\ds\frac{L_3}{2^{n-1}},\frac{1}{2^n}\left(\frac{4L_3}{\delta_1}+3K_3\right),\frac{3L_3}{2^{n(1-\alpha)-2}},\frac{C^{2\beta-\alpha}}{2^{(n-1)\beta(\beta-\alpha)-1}}\right\}\leq\ds\frac{\xi}{12}$, where $L_3$ is a Lipschitz constant of $f_3$ and $f_3^{-1}$ in the second variable and $K_3$ is a Lipschitz constant of $\ds\frac{\partial f_3}{\partial x}$ with respect to $y$;
\end{enumerate}

then $\tilde{f}\in\F_{(C,\beta)}$ and $d_\alpha(\tilde{f},f_3)<\xi/3$.

\end{lemma}

\begin{proof}

\begin{enumerate}
\item
Clearly, $\tilde{f}$ is $C^1$ in the first variable.
\item
If $y\in\left[\ds\frac{2a}{2^n},\frac{2a+1}{2^n}\right]$, $y=\ds\frac{2a+t}{2^n}$, for some $t\in[0,1]$, then
\begin{align*}
\tilde{f}(0,y)=\tilde{f}\left(0,\frac{2a+t}{2^n}\right)&=(1-t\tilde{a}(0))f_3\left(0,\ds\frac{2a}{2^n}\right)+t\tilde{a}(0)f_3\left(0,\frac{2a+2}{2^n}\right) \\
&=\left(1-\frac{t}{2}\right)\cdot\frac{2a}{2^n}+\frac{t}{2}\cdot\frac{2a+2}{2^n} \\
&=\frac{2a+t}{2^n} \\
&=y.
\end{align*}
Same for $y\in\left[\ds\frac{2a+1}{2^n},\frac{2a+2}{2^n}\right]$.
\item
If $x\in[0,1]$, $a=0$ and $t=0$, then $$\tilde{f}(x,0)=(1-0(1-a(x)))f_3(x,0)+0(1-a(x))f_3(x,0)=f_3(x,0)=0.$$
\item
If $x\in[0,1]$, $a=2^{n-1}-1$ and $t=2$, then
\begin{align*}
\tilde{f}(x,1)&=(1-\tilde{a}(x))(2-2)f_3\left(x,\frac{2^n-2}{2^n}\right)+(2-1+\tilde{a}(x)(2-2))f_3\left(x,\frac{2^n}{2^n}\right) \\
&=f_3(x,1) \\
&=1
\end{align*}
\item
If $y_1,y_2\in\left[\ds\frac{2a}{2^n},\frac{2a+1}{2^n}\right]$, then $y_1=\ds\frac{2a+t_1}{2^n}$, $y_2=\ds\frac{2a+t_2}{2^n}$, where $t_1,t_2\in[0,1]$. Suppose that $t_2>t_1$, then
\begin{align*}
|\tilde{f}(x,y_2)-\tilde{f}(x,y_1)|&=\left|(1-t_2\tilde{a}(x))f_3\left(x,\frac{2a}{2^n}\right)+t_2\tilde{a}(x)f_3\left(x,\frac{2a+2}{2^n}\right)\right. \\
&\mbox{ \ \ }\left.-(1-t_1\tilde{a}(x))f_3\left(x,\frac{2a}{2^n}\right)-t_1\tilde{a}(x)f_3\left(x,\frac{2a+2}{2^n}\right)\right| \\
&=\left|\left(f_3\left(x,\frac{2a+2}{2^n}\right)-f_3\left(x,\frac{2a}{2^n}\right)\right)\tilde{a}(x)(t_2-t_1)\right| \\
&\leq L_3\cdot\frac{2}{2^{n}}\cdot(1-\delta_2)\cdot2^n\cdot|y_2-y_1| \\
&=2L_3(1-\delta_2)\cdot|y_2-y_1|^{1-\beta}\cdot|y_2-y_1|^\beta \\
&\leq 2L_3(1-\delta_2)\cdot\left(\frac{1}{2^n}\right)^{1-\beta}|y_2-y_1|^\beta \\
&=\frac{L_3(1-\delta_2)}{2^{n(1-\beta)-1}}|y_2-y_1|^\beta.
\end{align*}

We have the same for $y_1,y_2\in\left[\ds\frac{2a+1}{2^n},\ds\frac{2a+2}{2^n}\right]$. Now, if $y_1\in\left[\ds\frac{2a}{2^n},\ds\frac{2a+1}{2^n}\right]$, $y_2\in\left[\ds\frac{2a+1}{2^n},\ds\frac{2a+2}{2^n}\right]$ and using the fact that $a^\beta+b^\beta<2(a+b)^\beta$, for $a,b>0$ and $0<\beta<1$, then

\begin{align*}
|\tilde{f}(x,y_2)-\tilde{f}(x,y_1)|&=\left|\tilde{f}(x,y_2)-\tilde{f}\left(x,\frac{2a+1}{2^n}\right)\right|+\left|\tilde{f}\left(x,\frac{2a+1}{2^n}\right)-\tilde{f}(x,y_1)\right| \\
&\leq \frac{L_3(1-\delta_2)}{2^{n(1-\beta)-1}}\left|y_2-\frac{2a+1}{2^n}\right|^\beta+\frac{L_3(1-\delta_2)}{2^{n(1-\beta)-1}}\left|\frac{2a+1}{2^n}-y_1\right|^\beta \\
&<\frac{L_3(1-\delta_2)}{2^{n(1-\beta)-2}}|y_2-y_1|^\beta.
\end{align*}

Now, if $y_1\in\left[\ds\frac{2a_1}{2^n},\ds\frac{2a_1+2}{2^n}\right]$, $y_2\in\left[\ds\frac{2a_2}{2^n},\ds\frac{2a_2+2}{2^n}\right]$, for $a_2>a_1$ and recalling that $\tilde{f}\left(x,\ds\frac{2a_2}{2^n}\right)=f_3\left(x,\ds\frac{2a_2}{2^n}\right)$ and $\tilde{f}\left(x,\ds\frac{2a_1+2}{2^n}\right)=f_3\left(x,\ds\frac{2a_1+2}{2^n}\right)$, then

\begin{align*}
|\tilde{f}(x,y_2)-\tilde{f}(x,y_1)|&=\left|\tilde{f}(x,y_2)-\tilde{f}\left(x,\frac{2a_2}{2^n}\right)\right|+\left|\tilde{f}\left(x,\frac{2a_2}{2^n}\right)-\tilde{f}\left(x,\frac{2a_1+2}{2^n}\right)\right| \\
&\mbox{ \ \ }+\left|\tilde{f}\left(x,\frac{2a_1+2}{2^n}\right)-\tilde{f}(x,y_1)\right| \\
&\leq \frac{L_3(1-\delta_2)}{2^{n(1-\beta)-2}}\left|y_2-\frac{2a_2}{2^n}\right|^\beta+C_\varepsilon\left|\frac{2a_2}{2^n}-\frac{2a_1+2}{2^n}\right|^\beta \\
&\mbox{ \ \ }+\frac{L_3(1-\delta_2)}{2^{n(1-\beta)-2}}\left|\frac{2a+2}{2^n}-y_1\right|^\beta \\
&<\frac{L_3(1-\delta_2)}{2^{n(1-\beta)-3}}\left|y_2-\frac{2a_2}{2^n}+\frac{2a+2}{2^n}-y_1\right|^\beta+C_\varepsilon\left|\frac{2a_2}{2^n}-\frac{2a+2}{2^n}\right|^\beta.
\end{align*}
So we have two cases: if $a_2=a_1+1$, then $$|\tilde{f}(x,y_2).\tilde{f}(x,y_1)|<\frac{L_3(1-\delta_2)}{2^{n(1-\beta)-3}}|y_2-y_1|^\beta,$$ but if $a_2>a_1+1$, then 

\begin{align*}
|\tilde{f}(x,y_2)-\tilde{f}(x,y_1)|&\leq\frac{L_3(1-\delta_2)}{2^{n(1-\beta)-3}}\left|y_2-\frac{2a_2}{2^n}+\frac{2a_1+2}{2^n}-y_1\right|^\beta+C_\varepsilon\left|\frac{2a_2}{2^n}-\frac{2a_1+2}{2^n}\right|^\beta \\
&<\frac{L_3(1-\delta_2)}{2^{n(1-\beta)-3}}|y_2-y_1|^\beta+C_\varepsilon|y_2-y_1|^\beta \\
&=\left(\frac{L_3(1-\delta_2)}{2^{n(1-\beta)-3}}+C_\varepsilon\right)|y_2-y_1|^\beta \\
&\leq C|y_2-y_1|^\beta,
\end{align*} 

and this implies that for all $y_1,y_2\in[0,1]$, $|\tilde{f}(x,y_2)-\tilde{f}(x,y_1)|\leq C|y_2-y_1|^\beta$.

\newpage

\item
If $y_1,y_2\in\left[\ds\frac{2a}{2^n},\frac{2a+1}{2^n}\right]$, then $y_1=\ds\frac{2a+t_1}{2^n}$, $y_2=\ds\frac{2a+t_2}{2^n}$, where $t_1,t_2\in[0,1]$. Suppose that $t_2>t_1$, then

\begin{align*}
|\tilde{f}(x,y_2)-\tilde{f}(x,y_1)|&=\left|\left(f_3\left(x,\frac{2a+2}{2^n}\right)-f_3\left(x,\frac{2a}{2^n}\right)\right)\tilde{a}(x)(t_2-t_1)\right| \\
&\geq L_3^{-1}\cdot\frac{2}{2^n}\cdot\delta_2\cdot 2^n|y_2-y_1| \\
&=2L_3^{-1}\delta_2|y_2-y_1|^{1-1/\beta}|y_2-y_1|^{1/\beta} \\
&\geq2L_3^{-1}\delta_2 \cdot\left(\frac{1}{2^n}\right)^{1-1/\beta}|y_2-y_1|^{1/\beta} \\
&=\frac{L_3^{-1}\delta_2}{2^{n(1-1/\beta)-1}}|y_2-y_1|^{1/\beta}
\end{align*}

Same for $y_1,y_2\in\left[\ds\frac{2a+1}{2^n},\ds\frac{2a+2}{2^n}\right]$. Now, if $y_1\in\left[\ds\frac{2a}{2^n},\ds\frac{2a+1}{2^n}\right]$, $y_2\in\left[\ds\frac{2a+1}{2^n},\ds\frac{2a+2}{2^n}\right]$ and using the fact that $a^{1/\beta}+b^{1/\beta}>\ds\frac{1}{2}(a+b)^{1/\beta}$, for $a,b>0$ and $0<\beta<1$, then

\begin{align*}
|\tilde{f}(x,y_2)-\tilde{f}(x,y_1)|&=\left|\tilde{f}(x,y_2)-\tilde{f}\left(x,\frac{2a+1}{2^n}\right)\right|+\left|\tilde{f}\left(x,\frac{2a+1}{2^n}\right)-\tilde{f}(x,y_1)\right| \\
&\geq\frac{L_3^{-1}\delta_2}{2^{n(1-1/\beta)-1}}\left|y_2-\frac{2a+1}{2^n}\right|^{1/\beta}+\frac{L_3^{-1}\delta_2}{2^{n(1-1/\beta)-1}}\left|\frac{2a+1}{2^n}-y_1\right|^{1/\beta} \\
&\geq\frac{L_3^{-1}\delta_2}{2^{n(1-1/\beta)}}|y_2-y_1|^{1/\beta}.
\end{align*}

Now, if $y_1\in\left[\ds\frac{2a_1}{2^n},\ds\frac{2a_1+2}{2^n}\right]$, $y_2\in\left[\ds\frac{2a_2}{2^n},\ds\frac{2a_2+2}{2^n}\right]$, for $a_2>a_1$ and recalling that $\tilde{f}\left(x,\ds\frac{2a_2}{2^n}\right)=f_3\left(x,\ds\frac{2a_2}{2^n}\right)$ and $\tilde{f}\left(x,\ds\frac{2a_1+2}{2^n}\right)=f_3\left(x,\ds\frac{2a_1+2}{2^n}\right)$, then
\begin{align*}
|y_2-y_1|&=\left|y_2-\frac{2a_2}{2^n}\right|+\left|\frac{2a_2}{2^n}-\frac{2a_1+2}{2^n}\right|+\left|\frac{2a_1+2}{2^n}-y_1\right| \\
&\leq\left(\frac{2^{n(1-1/\beta)}}{L_3^{-1}\delta_2}\right)^\beta\left|\tilde{f}(x,y_2)-\tilde{f}\left(x,\frac{2a_2}{2^n}\right)\right|^\beta+C_\varepsilon^\beta\left|f_3\left(x,\frac{2a_2}{2^n}\right)-f_3\left(x,\frac{2a_1+2}{2^n}\right)\right|^\beta \\
&\mbox{ \ \ }+\left(\frac{2^{n(1-1/\beta)}}{L_3^{-1}\delta_2}\right)^\beta\left|\tilde{f}\left(x,\frac{2a_1+2}{2^n}\right)-\tilde{f}(x,y_1)\right|^\beta \\
&\leq2\left(\frac{2^{n(1-1/\beta)}}{L_3^{-1}\delta_2}\right)^\beta\left|\tilde{f}(x,y_2)-\tilde{f}\left(x,\frac{2a_2}{2^n}\right)+\tilde{f}\left(x,\frac{2a_1+2}{2^n}\right)-\tilde{f}(x,y_1)\right|^\beta \\
&\mbox{ \ \ }+2\left(\frac{2^{n(1-1/\beta)}}{L_3^{-1}\delta_2}\right)^\beta\left|\tilde{f}\left(x,\frac{2a_2}{2^n}\right)-\tilde{f}\left(x,\frac{2a_1+2}{2^n}\right)\right|^\beta \\
&\leq4\left(\frac{2^{n(1-1/\beta)}}{L_3^{-1}\delta_2}\right)^\beta|\tilde{f}(x,y_2)-\tilde{f}(x,y_1)|^\beta \\
&=\left(\frac{2^{n(1-1/\beta)+2/\beta}}{L_3^{-1}\delta_2}\right)^\beta|\tilde{f}(x,y_2)-\tilde{f}(x,y_1)|^\beta \\
&\leq C^\beta|\tilde{f}(x,y_2)-\tilde{f}(x,y_1)|^\beta.
\end{align*}
\end{enumerate}

\ \

So for all $y_1,y_2\in[0,1]$, $|\tilde{f}(x,y_2)-\tilde{f}(x,y_1)|\geq C^{-1}|y_2-y_1|^{1/\beta}$. With all of this, $\tilde{f}\in\F_{(C,\beta)}$. Now we have to check that $d_{\alpha}(\tilde{f},f_3)<\xi/3$.
\begin{enumerate}
\item
We are going to estimate $\|\tilde{f}-f_3\|_{C^0}$. If $y\in[0,1]$, there exists $a=0,\ldots,2^{n-1}$ such that $y\in\left[\ds\frac{2a}{2^n},\frac{2a+2}{2^n}\right]$, $y=\ds\frac{2a+t}{2^n}$, for some $t\in[0,2]$. So, 
$$|\tilde{f}(x,y)-f_3(x,y)|\leq\left|f_3\left(x,\frac{2a+2}{2^n}\right)-f_3\left(x,\frac{2a}{2^n}\right)\right|\leq\frac{2L_3}{2^n},$$ so this implies that $\|\tilde{f}-f_3\|_{C^0}\leq\ds\frac{L_3}{2^{n-1}}\leq\ds\frac{\xi}{12}$.
\item
We are going to estimate $\left\|\ds\frac{\partial \tilde{f}}{\partial x}-\frac{\partial f_3}{\partial x}\right\|_{C^0}$. If $y\in\left[\ds\frac{2a}{2^n},\frac{2a+1}{2^n}\right]$, then
\begin{align*}
\left|\frac{\partial\tilde{f}}{\partial x}(x,y)-\frac{\partial f_3}{\partial x}(x,y)\right|&=\left|-t\tilde{a}'(x)f_3\left(x,\frac{2a}{2^n}\right)+(1-t\tilde{a}(x))\frac{\partial f_3}{\partial x}\left(x,\frac{2a}{2^n}\right)+t\tilde{a}'(x)f_3\left(x,\frac{2a+2}{2^n}\right)\right. \\
&\mbox{ \ \ }\left.+t\tilde{a}(x)\frac{\partial f_3}{\partial x}\left(x,\frac{2a+2}{2^n}\right)-\frac{\partial f_3}{\partial x}(x,y)\right| \\
&=\left|t\tilde{a}'(x)\left(f_3\left(x,\frac{2a+2}{2^n}\right)-f_3\left(x,\frac{2a}{2^n}\right)\right)+\frac{\partial f_3}{\partial x}\left(x,\frac{2a}{2^n}\right)-\frac{\partial f_3}{\partial x}(x,y)\right. \\
&\mbox{ \ \ }\left.+t\tilde{a}(x)\left(\frac{\partial f_3}{\partial x}\left(x,\frac{2a+2}{2^n}\right)-\frac{\partial f_3}{\partial x}\left(x,\frac{2a}{2^n}\right)\right)\right| \\
&\leq t|\tilde{a}'(x)|\left|f_3\left(x,\frac{2a+2}{2^n}\right)-f_3\left(x,\frac{2a}{2^n}\right)\right| +\left|\frac{\partial f_3}{\partial x}(x,y)-\frac{\partial f_3}{\partial x}\left(x,\frac{2a}{2^n}\right)\right|\\
&\mbox{ \ \ }+t\tilde{a}(x)\left|\frac{\partial f_3}{\partial x}\left(x,\frac{2a+2}{2^n}\right)-\frac{\partial f_3}{\partial x}\left(x,\frac{2a}{2^n}\right)\right| \\
&\leq \frac{2(1-2\delta_2)}{\delta_1}\cdot L_3\cdot\frac{2}{2^n}+K_3\cdot\frac{1}{2^n}+(1-\delta_2)K_3\cdot\frac{2}{2^n}\\
&=\frac{1}{2^n}\left(\frac{4L_3(1-2\delta_2)}{\delta_1}+2K_3(1-\delta_2)+K_3\right) \\
&<\frac{1}{2^n}\left(\frac{4L_3}{\delta_1}+3K_3\right).
\end{align*}

Same for $y\in\left[\ds\frac{2a+1}{2^n},\frac{2a+2}{2^n}\right]$. So, 
$$\left\|\ds\frac{\partial\tilde{f}}{\partial x}-\ds\frac{\partial f_3}{\partial x}\right\|_{C^0}\leq\frac{1}{2^n}\left(\frac{4L_3}{\delta_1}+3K_3\right)\leq\frac{\xi}{12}.$$
\item
We are going to estimate $\|h_{0,x}^{\tilde{f}}-h_{0,x}^{f_3}\|_{\alpha}$. For simplicity define $g(x,y):=\tilde{f}(x,y)-f_3(x,y)$. So, if $y\in\left[\ds\frac{2a}{2^n},\ds\frac{2a+1}{2^n}\right]$, then

\begin{align*}
g(x,y)&=\tilde{f}(x,y)-f_3(x,y) \\
&=(1-t\tilde{a}(x))f_3\left(x,\frac{2a}{2^n}\right)+t\tilde{a}(x)f_3\left(x,\frac{2a+2}{2^n}\right)-f_3(x,y) \\
&=t\tilde{a}(x)\left(f_3\left(x,\frac{2a+2}{2^n}\right)-f_3\left(x,\frac{2a}{2^n}\right)\right)-\left(f_3(x,y)-f_3\left(x,\frac{2a}{2^n}\right)\right),
\end{align*}

and if $y\in\left[\ds\frac{2a+1}{2^n},\ds\frac{2a+2}{2^n}\right]$, then

\begin{align*}
g(x,y)&=\tilde{f}(x,y)-f_3(x,y) \\
&=(1-\tilde{a}(x))(2-t)f_3\left(x,\frac{2a}{2^n}\right)+(t-1+\tilde{a}(x)(2-t))f_3\left(x,\frac{2a+2}{2^n}\right)-f_3(x,y) \\
&=(\tilde{a}(x)(2-t)+t-1)\left(f_3\left(x,\frac{2a+2}{2^n}\right)-f_3\left(x,\frac{2a}{2^n}\right)\right)-\left(f_3(x,y)-f_3\left(x,\frac{2a}{2^n}\right)\right)
\end{align*}

In the first case, if $y_1,y_2\in\left[\ds\frac{2a}{2^n},\ds\frac{2a+1}{2^n}\right]$, then

\begin{align*}
|g(x,y_2)-g¡(x,y_1)|&=\left|t_2\tilde{a}(x)\left(f_3\left(x,\frac{2a+2}{2^n}\right)-f_3\left(x,\frac{2a}{2^n}\right)\right)-\left(f_3(x,y_2)-f_3\left(x,\frac{2a}{2^n}\right)\right)\right. \\
&\mbox{ \ \ }\left.-t_1\tilde{a}(x)\left(f_3\left(x,\frac{2a+2}{2^n}\right)-f_3\left(x,\frac{2a}{2^n}\right)\right)+\left(f_3(x,y_1)-f_3\left(x,\frac{2a}{2^n}\right)\right)\right| \\
&=\left|\tilde{a}(x)(t_2-t_1)\left(f_3\left(x,\frac{2a+2}{2^n}\right)-f_3\left(x,\frac{2a}{2^n}\right)\right)-(f_3(x,y_2)-f_3(x,y_1))\right| \\
&\leq(1-\delta_2)\cdot 2^n\cdot |y_2-y_1|\cdot L_3\cdot\frac{2}{2^n}+L_3|y_2-y_1| \\
&=(2L_3(1-\delta_2)+L_3)|y_2-y_1| \\
&<3L_3|y_2-y_1|^{1-\alpha}\cdot|y_2-y_1|^\alpha \\
&\leq\frac{3L_3}{2^{n(1-\alpha)}}|y_2-y_1|^\alpha.
\end{align*}

Same for $y_1,y_2\in\left[\ds\frac{2a+1}{2^n},\ds\frac{2a+2}{2^n}\right]$. Now, if $y_1\in\left[\ds\frac{2a}{2^n},\ds\frac{2a+1}{2^n}\right]$ and $y_2\in\left[\ds\frac{2a+1}{2^n},\ds\frac{2a+2}{2^n}\right]$, then

\begin{align*}
|g(x,y_2)-g(x,y_1)|&\leq\left|g(x,y_2)-g\left(x,\frac{2a+1}{2^n}\right)\right|+\left|g\left(x,\frac{2a+1}{2^n}\right)-g(x,y_1)\right| \\
&<\frac{3L_3}{2^{n(1-\alpha)}}\left|y_2-\frac{2a+1}{2^n}\right|^\alpha+\frac{3L_3}{2^{n(1-\alpha)}}\left|\frac{2a+1}{2^n}-y_1\right|^\alpha \\
&<\frac{3L_3}{2^{n(1-\alpha)-1}}|y_2-y_1|^\alpha.
\end{align*}

For the last case, note that $g\left(x,\ds\frac{2a}{2^n}\right)=0$, for all $a\in\{0,1,\ldots,2^n-1\}$, so if $y_1\in\left[\ds\frac{2a}{2^n},\ds\frac{2a+2}{2^n}\right]$ and $y_2\in\left[\ds\frac{2a_2}{2^n},\ds\frac{2a_2+2}{2^n}\right]$, for $a_2>a$, then

\begin{align*}
|g(x,y_2)-g(x,y_1)|&\leq\left|g(x,y_2)-g\left(x,\frac{2a_2}{2^n}\right)\right|+\left|g\left(x,\frac{2a_2}{2^n}\right)-g\left(x,\frac{2a+2}{2^n}\right)\right| \\
&\mbox{ \ \ }+\left|g\left(x,\frac{2a+2}{2^n}\right)-g(x,y_1)\right| \\
&<\frac{3L_3}{2^{n(1-\alpha)-1}}\cdot\left|y_2-\frac{2a_2}{2^n}\right|^\alpha+\frac{3L_3}{2^{n(1-\alpha)-1}}\cdot\left|\frac{2a+2}{2^n}-y_1\right|^\alpha \\
&<\frac{3L_3}{2^{n(1-\alpha)-2}}\cdot\left|y_2-\frac{2a_2}{2^n}+\frac{2a+2}{2^n}-y_1\right|^\alpha \\
&<\frac{3L_3}{2^{n(1-\alpha)-2}}|y_2-y_1|^\alpha.
\end{align*}

Finally, we have that $\|h_{0,x}^{\tilde{f}}-h_{0,x}^{f_3}\|_{\alpha}\leq\ds\frac{3L_3}{2^{n(1-\alpha)-2}}\leq\frac{\xi}{12}$.

\item
For simplicity, define $g(x,y):=\tilde{f}(x,\cdot)^{-1}(y)-f_3(x,\cdot)^{-1}(y):=y^{\tilde{f}}-y^{f_3}$. In a first case, take $y_1,y_2\in\left[\tilde{f}\left(x,\ds\frac{2a+1}{2^n}\right),\tilde{f}\left(x,\ds\frac{2a+2}{2^n}\right)\right]$ and note that $$0<y_2-y_1\leq\left|\tilde{f}\left(x,\frac{2a+2}{2^n}\right)-\tilde{f}\left(x,\frac{2a}{2^n}\right)\right|\leq C\left(\frac{2}{2^n}\right)^\beta.$$ This implies that

\begin{align*}
|g(x,y_2)-g(x,y_1)|&=|\tilde{f}(x,\cdot)^{-1}(y_2)-f_3(x,\cdot)^{-1}(y_2)-\tilde{f}(x,\cdot)^{-1}(y_1)+f_3(x,\cdot)^{-1}(y_1)| \\
&:=|y_2^{\tilde{f}}-y_2^{f_3}-y_1^{\tilde{f}}+y_1^{f_3}| \\
&=|y_2^{\tilde{f}}-y_1^{\tilde{f}}-(y_2^{f_3}-y_1^{f_3})| \\
&<|y_2^{\tilde{f}}-y_1^{\tilde{f}}| \\
&\leq C^\beta|y_2-y_1|^\beta \\
&=C^\beta|y_2-y_1|^{\beta-\alpha}|y_2-y_1|^\alpha \\
&\leq C^\beta\cdot C^{\beta-\alpha}\cdot\left(\frac{2}{2^n}\right)^{\beta(\beta-\alpha)}|y_2-y_1|^\alpha \\
&=\frac{C^{2\beta-\alpha}}{2^{(n-1)\beta(\beta-\alpha)}}|y_2-y_1|^\alpha
\end{align*}

Same for $y_1,y_2\in\left[\tilde{f}\left(x,\ds\frac{2a+1}{2^n}\right),\tilde{f}\left(x,\ds\frac{2a+2}{2^n}\right)\right]$. Now, if \newline $y_1\in\left[\tilde{f}\left(x,\ds\frac{2a_1}{2^n}\right),\tilde{f}\left(x,\ds\frac{2a_1+2}{2^n}\right)\right]$ and $y_2\in\left[\tilde{f}\left(x,\ds\frac{2a_2}{2^n}\right),\tilde{f}\left(x,\ds\frac{2a_2+2}{2^n}\right)\right]$ with $a_2>a_1$, consider $z_1,z_2$ such that $f_3\left(x,\ds\frac{2a_1+2}{2^n}\right)=z_1$ and $f_3\left(x,\ds\frac{2a_2}{2^n}\right)=z_2$, then 

\begin{align*}
|g(x,y_2)-g(x,y_1)|&\leq|g(x,y_2)-g(x,z_2)|+|g(x,z_2)-g(x,z_1)|+|g(x,z_1)-g(x,y_1)| \\
&\leq\frac{C^{2\beta-\alpha}}{2^{(n-1)\beta(\beta-\alpha)}}|y_2-z_2|^\alpha+\frac{C^{2\beta-\alpha}}{2^{(n-1)\beta(\beta-\alpha)}}|z_1-y_1|^\alpha \\
&<\frac{C^{2\beta-\alpha}}{2^{(n-1)\beta(\beta-\alpha)-1}}|y_2-y_1-(z_2-z_1)|^\alpha \\
&<\frac{C^{2\beta-\alpha}}{2^{(n-1)\beta(\beta-\alpha)-1}}|y_2-y_1|^\alpha.
\end{align*}

Finally, we have that $\|h_{x,0}^{\tilde{f}}-h_{x,0}^{f_3}\|_\alpha\leq\ds\frac{C^{2\beta-\alpha}}{2^{(n-1)\beta(\beta-\alpha)-1}}\leq\frac{\xi}{12}$, and all of this implies that
\begin{align*}
d_\alpha(\tilde{f},f_3)&=\|\tilde{f}-f_3\|_{C^0}+\left\|\frac{\partial \tilde{f}}{\partial x}-\frac{\partial f_3}{\partial x}\right\|_{C^0}+\|h_{0,x}^{\tilde{f}}-h_{0,x}^{f_3}\|_{\alpha}+\|h_{x,0}^{\tilde{f}}-h_{x,0}^{f_3}\|_{\alpha} \\
&\leq \frac{L_3}{2^{n-1}}+\frac{1}{2^n}\left(\frac{4L_3(1-2\delta_2)}{\delta_1}+2K_3 +L_3\right)+\frac{L_3(1-\delta_2)}{2^{n(1-\alpha)-2}} +\frac{C^{2\beta-\alpha}}{2^{(n-1)\beta(\beta-\alpha)-1}} \\
&\leq\frac{\xi}{12}+\frac{\xi}{12}+\frac{\xi}{12}+\frac{\xi}{12} \\
&=\frac{\xi}{3}.
\end{align*}
\end{enumerate}
\end{proof}

\begin{lemma} \label{lema5} Furthermore, if $\delta_1,\delta_2$ satisfy $$\max\left\{\frac{L_3^2\delta_2\mu_1(I)}{1-\mu_1(I)-3\delta_1},\frac{L_3^2(3\delta_1+\delta_2(1-\mu_1(I)-3\delta_1))}{\mu_1(I)}\right\}<\frac{1}{m},$$ where $\mu_1$ is the Lebesgue measure in $[0,1]$, then $\tilde{f}\in B_{m,I}$.
\end{lemma}

\begin{proof}

Let define the sets 
$$P_1=\left\{(x,y)\in M:y\in\left[\tilde{f}\left(x,\frac{2a}{2^n}\right),\tilde{f}\left(x,\frac{2a+1}{2^n}\right)\right)\right\}$$ 
and 
$$P_2=\left\{(x,y)\in M:y\in\left[\tilde{f}\left(x,\frac{2a+1}{2^n}\right),\tilde{f}\left(x,\frac{2a+2}{2^n}\right)\right)\right\}.$$

\begin{enumerate}
\item
First we have
\begin{align*}
\mu(P_1)&\geq(1-\delta_2)\int_{\delta_1}^{b_1-\delta_1}\left(f_3\left(x,\frac{2a+2}{2^n}\right)-f_3\left(x,\frac{2a}{2^n}\right)\right)dx \\
&\mbox{ \ \ }+(1-\delta_2)\int_{b_2+\delta_1}^1\left(f_3\left(x,\frac{2a+2}{2^n}\right)-f_3\left(x,\frac{2a}{2^n}\right)\right)dx \\
&\geq (1-\delta_2)(b_1-2\delta_1)\cdot L_3^{-1}\cdot\frac{2}{2^n}+(1-\delta_2)(1-b_2-\delta_1)\cdot L_3^{-1}\cdot\frac{2}{2^n} \\
&=\frac{L_3^{-1}}{2^{n-1}}(1-\mu_1(I)-3\delta_1).
\end{align*}
Now, 
$$\mu(P_1\cap\tilde{I})=\delta_2\int_{b_1}^{b_2}\left(f_3\left(x,\frac{2a+2}{2^n}\right)-f_3\left(x,\frac{2a}{2^n}\right)\right)dx\leq\frac{2L_3\delta_2\mu_1(I)}{2^n}.$$
So, $$\frac{\mu(P_1\cap\tilde{I})}{\mu(P_1)}\leq\frac{L_3^2\delta_2\mu_1(I)}{1-\mu_1(I)-3\delta_1}<\frac{1}{m}.$$
\item
Let define $C_2=P_2\setminus (P_2\cap\tilde{I})$, then 
$$\frac{\mu(C_2)}{\mu(P_2)}=1-\frac{\mu(P_2\cap\tilde{I})}{\mu(P_2)}.$$
So, if we want that $\ds\frac{\mu(P_2\cap\tilde{I})}{\mu(P_2)}>1-\frac{1}{m}$, is enough to see that $\ds\frac{\mu(C_2)}{\mu(P_2)}<\frac{1}{m}$. First, we have
\begin{align*}
\mu(C_2)&\leq\int_0^{\delta_1}\left(f_3\left(x,\frac{2a+2}{2^n}\right)-f_3\left(x,\frac{2a}{2^n}\right)\right)dx+\delta_2\int_{\delta_1}^{b_1-\delta_1}\left(f_3\left(x,\frac{2a+2}{2^n}\right)-f_3\left(x,\frac{2a}{2^n}\right)\right)dx \\
&\mbox{ \ \ }+\int_{b_1-\delta_!}^{b_1} \left(f_3\left(x,\frac{2a+2}{2^n}\right)-f_3\left(x,\frac{2a}{2^n}\right)\right)dx+\int_{b_2}^{b_2+\delta_1}\left(f_3\left(x,\frac{2a+2}{2^n}\right)-f_3\left(x,\frac{2a}{2^n}\right)\right)dx \\
&\mbox{ \ \ }+\delta_2\int_{b_2+\delta_1}^1 \left(f_3\left(x,\frac{2a+2}{2^n}\right)-f_3\left(x,\frac{2a}{2^n}\right)\right)dx \\
&\leq\delta_1\cdot L_3\cdot\frac{2}{2^n}+\delta_2(b_1-2\delta_1)L_3\cdot\frac{2}{2^n}+\delta_1\cdot L_3\cdot\frac{2}{2^n}+\delta_2(1-b_2-\delta_1)\cdot L_3\cdot \frac{2}{2^n} \\
&=\frac{L_3}{2^{n-1}}(3\delta_1+\delta_2(1-\mu_1(I)-3\delta_1)).
\end{align*}

Now,
\begin{align*}
\mu(P_2)&\geq(1-\delta_2)\int_{b_1}^{b_2}\left(f_3\left(x,\frac{2a+2}{2^n}\right)-f_3\left(x,\frac{2a}{2^n}\right)\right)dx \\
&\geq(1-\delta_2)\mu_1(I)\cdot L_3^{-1}\cdot\frac{2}{2^n} \\
&=\frac{L_3^{-1}}{2^{n-1}}.
\end{align*}

So, 
$$\frac{\mu(C_2)}{\mu(P_2)}\leq\frac{L_3^2 (3\delta_1+\delta_2(1-\mu_1(I)-3\delta_1))}{\mu_1(I)}<\frac{1}{m}.$$
\end{enumerate}

Therefore, $\tilde{f}\in B_{m,I}$.
\end{proof}

To sum up, given $\delta_1,\delta_2>0$ satisfying the hypothesis of Lemma \ref{lema5}, there exists $n\in\N$ satisfying the hypothesis of Lemma \ref{lema4} such that the corresponding $\tilde{f}$ satisfies $d_{\alpha}(f,\tilde{f})<\xi$ and $\tilde{f}\in B_{m,I}$. So, $B_{m,I}$ is dense, for any $m\in\N$ and $I=[b_1,b_2]$, with $b_1,b_2\in\Q$.
\end{proof}

\ \
\ \

\section{Proof of Theorem \ref{main}}{\label{teo}}

\ \
\ \

In this section we present the proof of Theorem \ref{main}. Let define $\RR=\bigcap_m\bigcap_{I=[b,c],b,c\in\Q}B_{m,I}$.

\begin{proof} Let $f\in\RR$, then for all $m\in\N$, for all $I=[b_1,b_2]$ with $b_1,b_2\in\Q$, there exist $n\in\N$ such that $$\frac{\mu(P_n\cap\tilde{I})}{\mu(P_n)}<\frac{1}{m}\mbox{ or } \frac{\mu(P_n\cap\tilde{I})}{\mu(P_n)}>1-\frac{1}{m},$$ where $\tilde{I}=I\times[0,1]$ and $P_n\in\PP_{f,n}$. This implies that $\left\{\ds\frac{\mu(P_n\cap\tilde{I})}{\mu(P_n)}:n\in\N\right\}$ has $0$ or $1$ as an accumulation point. If $$\mu_P(\tilde{I})=\lim_{n\to\infty}\frac{\mu(P_n\cap\tilde{I})}{\mu(P_n)}.$$ exists, then it has to be $0$ or $1$. But by the observation after Rokhlin's theorem, for $\hat{\mu}$-a.e. $P\in\PP_f$, $$\mu_P(\tilde{I})=\lim_{n\to\infty}\frac{\mu(P_n\cap\tilde{I})}{\mu(P_n)}.$$ So, for $\hat{\mu}$-a.e. $P\in\PP_f$, $\mu_P(\tilde{I})=0$ or $1$. Note that if $P$ is a leaf of the foliation $\PP_f$, $P=f([0,1]\times\{y_P\})$, for some $y_P\in[0,1]$. So, the last statement says that $\mu_{y_P}(\tilde{I})=0$ or $1$, for $\mu_1$-a.e. $y_P\in[0,1]$. The next lemma help us to conclude the theorem:

\begin{lemma} If $\mu$ is a probability measure in $[0,1]$ such that $\mu(I)=0$ or $1$, for any $I=[b_1,b_2]$ with $b_1,b_2\in\Q$, then $\mu=\delta_x$, for some $x\in[0,1]$.
\end{lemma}

\begin{proof}
Assume that the statement is false, so there exist $x_1,x_2\in[0,1]$ with $x_1<x_2$ such that $x_1,x_2\in\supp(\mu)$. Is enough to consider two intervals $I_1=[b_1,b_2]$ and $I_2=[c_1,c_2]$ such that $I_1\cap I_2=\emptyset$, $x_1\in I_1$ and $x_2\in I_2$, $b_1,b_2,c_1,c_2\in\Q$.

\begin{figure}[h]
\begin{center}
\includegraphics[width=10cm]{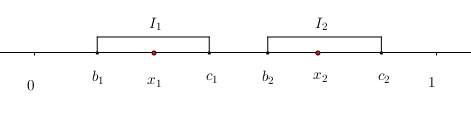}
\end{center}
\end{figure}

Since $x_1\in I_1$, $x_2\in I_2$ and $x_1,x_2\in\supp(\mu)$, then $\mu(I_1),\mu(I_2)>0$. So, for the condition of $\mu$ we have that $\mu(I_1)=\mu(I_2)=1$ and $I_1\cap I_2=\emptyset$, a contradiction.
\end{proof}

With this lemma, define the projection in the first variable $\pi_1:M\to[0,1]$ and the measure $\tilde{\mu}_P(A):=(\pi_1)_{*}\mu_{y_P}(A)=\mu_{y_P}(\pi_1^{-1}(A))$, for any measurable set $A\subset [0,1]$. So, if $\tilde{I}=I\times[0,1]$ with $I=[b_1,b_2]$, $b_1,b_2\in\Q$, then $\tilde{\mu}_{y_P}(I)=(\pi_1)_{*}\mu_{y_P}(I)=\mu_{y_P}(\pi_1^{-1}(I))=\mu_{y_P}(\tilde{I})=0$ or $1$. But for the lemma we have that $\tilde{\mu}_{y_P}=\delta_{x_P}$, for some $x_P\in[0,1]$. 

\newpage

\begin{figure}[h]
\begin{center}
\includegraphics[width=12cm]{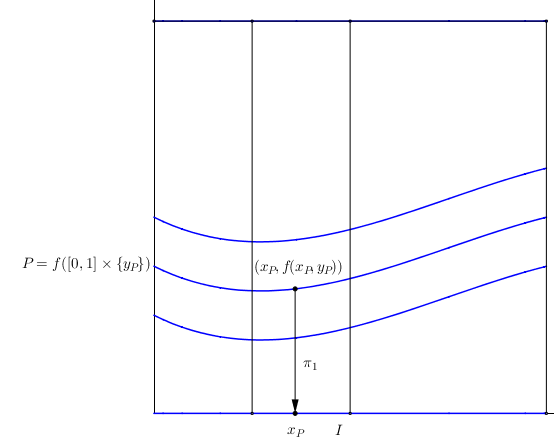}
\end{center}
\end{figure}

Finally, since $\supp(\mu_{y_P})\subset P$ and $(\pi_1)_{*}\mu_{y_P}=\delta_{x_P}$, we have $\mu_{y_P}=\delta_{(x_P,f(x_P,y_P))},$ for $\hat{\mu}$-a.e. $P\in\PP_f$.
\end{proof}

\begin{rem} The result could be extended to other spaces of foliations on the torus, in higher dimensions, etc.
\end{rem}

\textbf{Acknowledgments:} This research was supported by MATH AM-Sud Project 16-math-06 Physeco, PAI / Atracci\'on de capital humano avanzado del extranjero N$^{\circ}$PAI80160049 and FONDECYT Grant 1171477. Enzo Fuentes was partially supported by Beca CONICYT-PCHA / Doctorado Nacional / 2017-21170175.

\bibliographystyle{abbrv}
\bibliography{refe}

\end{document}